\documentclass[oneside,american,english,british]{amsart}
\usepackage[T1]{fontenc}
\usepackage{geometry}
\usepackage{enumerate}
\usepackage{amsthm}
\usepackage{amstext}
\usepackage{amssymb}
\usepackage{setspace}
\usepackage{units}
\usepackage{babel}
\usepackage[latin9,cp1255]{inputenc}
\PassOptionsToPackage{normalem}{ulem}
\usepackage{ulem}

\makeatletter
\numberwithin{equation}{section}
\numberwithin{figure}{section}

\theoremstyle{plain}
\newtheorem{thm}{\protect\theoremname}
  \theoremstyle{definition}
  \newtheorem{defn}[thm]{\protect\definitionname}
  \theoremstyle{definition}
  \newtheorem{problem}[thm]{\protect\problemname}
  \theoremstyle{plain}
  \newtheorem{prop}[thm]{\protect\propositionname}
  \theoremstyle{plain}
  
  \theoremstyle{plain}
  \newtheorem{cor}[thm]{\protect\corollaryname}
  \theoremstyle{remark}
  \newtheorem{rem}[thm]{\protect\remarkname}
  \theoremstyle{plain}
  \newtheorem{observation}[thm]{\protect\observationname}
  \theoremstyle{definition}
  
  \theoremstyle{remark}
  \newtheorem{claim}[thm]{\protect\claimname}
  \theoremstyle{plain}
  
\numberwithin{thm}{section}

\makeatother
\usepackage{babel}
  \addto\captionsamerican{\renewcommand{\claimname}{Claim}}
  \addto\captionsamerican{\renewcommand{\corollaryname}{Corollary}}
  \addto\captionsamerican{\renewcommand{\definitionname}{Definition}}
  \addto\captionsamerican{\renewcommand{\examplename}{Example}}
  \addto\captionsamerican{\renewcommand{\factname}{Fact}}
  \addto\captionsamerican{\renewcommand{\lemmaname}{Lemma}}
  \addto\captionsamerican{\renewcommand{\problemname}{Problem}}
  \addto\captionsamerican{\renewcommand{\propositionname}{Proposition}}
  \addto\captionsamerican{\renewcommand{\remarkname}{Remark}}
  \addto\captionsamerican{\renewcommand{\theoremname}{Theorem}}
  \addto\captionsbritish{\renewcommand{\claimname}{Claim}}
  \addto\captionsbritish{\renewcommand{\corollaryname}{Corollary}}
  \addto\captionsbritish{\renewcommand{\definitionname}{Definition}}
  \addto\captionsbritish{\renewcommand{\examplename}{Example}}
  \addto\captionsbritish{\renewcommand{\factname}{Fact}}
  \addto\captionsbritish{\renewcommand{\lemmaname}{Lemma}}
  \addto\captionsbritish{\renewcommand{\problemname}{Problem}}
  \addto\captionsbritish{\renewcommand{\propositionname}{Proposition}}
  \addto\captionsbritish{\renewcommand{\remarkname}{Remark}}
  \addto\captionsbritish{\renewcommand{\theoremname}{Theorem}}
  \addto\captionsenglish{\renewcommand{\claimname}{Claim}}
  \addto\captionsenglish{\renewcommand{\corollaryname}{Corollary}}
  \addto\captionsenglish{\renewcommand{\definitionname}{Definition}}
  \addto\captionsenglish{\renewcommand{\examplename}{Example}}
  \addto\captionsenglish{\renewcommand{\factname}{Fact}}
  \addto\captionsenglish{\renewcommand{\lemmaname}{Lemma}}
  \addto\captionsenglish{\renewcommand{\problemname}{Problem}}
  \addto\captionsenglish{\renewcommand{\propositionname}{Proposition}}
  \addto\captionsenglish{\renewcommand{\remarkname}{Remark}}
  \addto\captionsenglish{\renewcommand{\theoremname}{Theorem}}
  \addto\captionsenglish{\renewcommand{\observationname}{Observation}}
  \addto\captionsbritish{\renewcommand{\observationname}{Observation}}
  \addto\captionsamerican{\renewcommand{\observationname}{Observation}}
  \providecommand{\claimname}{Claim}
  \providecommand{\corollaryname}{Corollary}
  \providecommand{\definitionname}{Definition}
  \providecommand{\examplename}{Example}
  \providecommand{\factname}{Fact}
  \providecommand{\lemmaname}{Lemma}
  \providecommand{\problemname}{Problem}
  \providecommand{\propositionname}{Proposition}
  \providecommand{\remarkname}{Remark}
  \providecommand{\theoremname}{Theorem}
  \providecommand{\observationname}{Observation}

\begin{document}
\begin{onehalfspace}

\title{Measurability and
Perfect Set Theorems for Equivalence Relations with Small Classes}

\author{Ohad Drucker}
\maketitle
\begin{abstract}
We ask whether $\mathbf{\Delta^1_2}$ or $\mathbf{\Sigma^1_2}$ equivalence relations with $I$-small classes for $I$ a $\sigma$-ideal must have perfectly many classes. We show that for a wide class of ccc $\sigma$-ideals, a positive answer for $\mathbf{\Delta^1_2}$ equivalence relations is equivalent to the $I$-measurability of $\mathbf{\Delta^1_2}$ sets. However, the analogous statement for $\mathbf{\Sigma^1_2}$ equivalence relations is false:  $\mathbf{\Sigma^1_2}$ equivalence relations with meager classes have a perfect set of pairwise inequivalent elements if and only if $\mathbf{\Delta^1_2}$ sets have the Baire property.
\end{abstract}
\section{Introduction}

An equivalence relation $E$ on a Polish space $X$ is said to have
\textit{perfectly many classes} if there is a perfect set $P\subseteq X$
whose elements are pairwise inequivalent.

Given a
$\sigma$-ideal $I$, we say that a set $A$ is \emph{$I$-small} if
$A\in I$ and \emph{$I$-positive} if $A\notin I$. $\mathbb{P}_I$ is the partial order of $I$-positive Borel sets, ordered by inclusion. $I$ is \emph{proper} if $\mathbb{P}_I$ is.

A theorem due to Silver states that $\mathbf{\Pi^1_1}$ equivalence relations either have countably many classes or perfectly many classes. Therefore, a $\mathbf{\Pi^1_1}$ equivalence relation on a Borel $I$-positive set whose classes are $I$-small must have perfectly many classes -- a property we will denote by $PSP_I(\mathbf{\Pi_{1}^{1})}:$

\begin{defn}
\label{definition PSP}For $I$ a $\sigma$-ideal and $\Gamma$ a pointclass, $PSP_{I}(\Gamma)$
(for ``Perfect Set Property'') is the following statement:
``If $E\in\Gamma$ is an equivalence relation on $B$ Borel $I$-positive with
$I$-small classes then $E$ has perfectly many classes''.
\end{defn}

Silver's theorem proving $PSP_I(\mathbf{\Pi_{1}^{1})}$ relied on the definition of the equivalence relation -- namely, its $\Pi^1_1$ definition. Other theorems use the measurability of the equivalence relation to arrive at the same conclusion:

\begin{thm}
\label{mycielski_meager} \cite{Mycielski_meager} If $E$ is an equivalence relation on a Borel nonmeager set that has
the Baire property, and all $E$-classes are meager, then $E$ has
perfectly many classes.
\end{thm}

\begin{thm}
\label{Mycielski_measure_zero}\cite{Mycielski_null} If $E$ is a Lebesgue measurable equivalence relation on a Borel set of positive measure and all $E$-classes are null, then
$E$ has perfectly many classes.
\end{thm}

In particular, since analytic sets are Lebesgue measurable and have the Baire property, Mycielski has shown $PSP_{meager}(\mathbf{\Sigma^1_1})$ and $PSP_{null}(\mathbf{\Sigma^1_1})$. Furthermore, in \cite{my_paper} we have shown that $PSP_{I}(\mathbf{\Sigma_{1}^{1})}$ is true
for any proper $\sigma$-ideal $I$.

This paper investigates $PSP_I(\mathbf{\Delta_{2}^{1}})$ and $PSP_I(\mathbf{\Sigma_{2}^{1}})$. We can use Mycielski's results above as a starting point -- they clearly imply:

\begin{observation}
\label{observation_mycielski}
\begin{enumerate}
\item If all $\mathbf{\Delta_{2}^{1}}$ sets are Lebesgue measurable (have the Baire property)
then  $PSP_{null}(\mathbf{\Delta_{2}^{1})}$ ($PSP_{meager}(\mathbf{\Delta_{2}^{1})}$). 
\item If all $\mathbf{\Sigma_{2}^{1}}$
sets are Lebesgue measurable (have the Baire property) then  $PSP_{null}(\mathbf{\Sigma_{2}^{1})}$ ($PSP_{meager}(\mathbf{\Sigma_{2}^{1})}$). 
\end{enumerate}
\end{observation}

In section \ref{sec:---measruable} we will see how measurability can be generalized to any $\sigma$-ideal. Then in light of the above observation we ask:

\begin{problem}
\label{problem}
Let $I$ be a $\sigma$-ideal.
\begin{enumerate}
\item Is $I$-measurability of all $\mathbf{\Delta^1_2}$ sets equivalent to $PSP_I(\mathbf{\Delta_{2}^{1}})$?
\item Is $I$-measurability of all $\mathbf{\Sigma^1_2}$ sets equivalent to $PSP_I(\mathbf{\Sigma_{2}^{1}})$?
\end{enumerate}
\end{problem}

\subsection{Previous results}

The main results on $PSP_I(\mathbf{\Delta^1_2})$ and $PSP_I(\mathbf{\Sigma^1_2})$ in \cite{my_paper} are:

\begin{thm}
\label{previous_delta_1_2} Let $I$ be a proper $\sigma$-ideal, and assume $\mathbf{\Sigma_{3}^{1}}$
$\mathbb{P}_{I}$-generic absoluteness. Then $PSP_{I}(\mathbf{\Delta_{2}^{1})}$.
\end{thm}

\begin{thm}
\label{thm_countable_ideal)}(countable ideal) The following are
equivalent:
\begin{enumerate}
\item $PSP_{countable}(\mathbf{\Delta_{2}^{1})}.$
\item For any real $z$, $\mathbb{R}^{L[z]}\neq\mathbb{R}.$
\end{enumerate}
\end{thm}

\begin{thm}
\label{previous_paper_thm_Sigma_1_2_meager}(meager ideal) If
for any real $z$ there is a Cohen real over $L[z]$, then $PSP_{I}(\mathbf{\Sigma_{2}^{1})}.$
\end{thm}

Theorem \ref{thm_countable_ideal)} serves as a positive evidence for problem \ref{problem} (1): in \cite{brendle_lowe} 7.1 it is shown that "for any real, $\mathbb{R}^{L[z]}\neq\mathbb{R}$" is equivalent to "all $\mathbf{\Delta^1_2}$ sets are measurable with respect to the countable ideal".

We note that in \cite{my_paper} the notion of $PSP_I$ refers to equivalence relations on all reals, whereas here it is a stronger notion referring to equivalence relations on Borel $I$-positive sets. However, all statements and proofs of \cite{my_paper} are valid for the stronger notion considered here, with the obvious changes in the proofs.

\subsection{Measurability, Generic Absoluteness and Transcendence over $L$}
\label{subsection_ikegami}
Judah
and Shelah \cite{ihoda-shelah-delta12} have shown that $\mathbf{\Delta_{2}^{1}}$ sets are Lebesgue
measurable if and only if for every $z$ there is a random real over
$L[z].$ In \cite{real_line_book} it is shown that
Lebesgue measurability of $\mathbf{\Delta_{2}^{1}}$ sets is
equivalent to $\Sigma_{3}^{1}$ random real generic absoluteness, which
is, $\Sigma_{3}^{1}$ statements are preserved under random real forcing.

The above results indicate a connection between measurability
of $\Delta_{2}^{1}$ sets, $\Sigma_{3}^{1}$-generic-absoluteness and transcendence over $L$ - namely, existence
of generics over $L$. This connection is not reserved
to the case of random real forcing - it exists for Cohen forcing
where measurability is replaced by the Baire property and random reals
by Cohen reals. It also exists for Sacks forcing \cite{ikegami_sacks} with the appropriate
generalizations of the notions of measurability and genericity. In
fact, Brendle and Lowe \cite{brendle_lowe} find similar equivalences for
most of the better known examples, whereas Ikegami in \cite{Ikegami_main} shows how to extend the above results to a wide
class of proper $\sigma$ - ideals.

The notion of $I$-measurability for a general $\sigma$-ideal is discussed in section \ref{sec:---measruable}. Here we list Ikegami's results, after translating them to the context of $I$-measurability we are working with in this paper.

A $\sigma$-ideal $I$ is said to be \emph{$\mathbf{\Sigma_{n}^{1}}$ or $\mathbf{\Pi^1_n}$}
if the set of Borel codes of $I$-small sets is. The term ``\textit{provably ccc}'' refers to $\sigma$-ideals which
are ccc in all models of $ZFC.$ An ideal is said to be \emph{Borel generated} if $I$-small sets are contained in $I$-small Borel sets. \emph{$\Sigma_{3}^{1}$-$\mathbb{P}_{I}$-generic-absoluteness} is the property that $\Sigma_{3}^{1}$ statements on ground model reals are absolute between the universe and $\mathbb{P}_I$-generic extensions of the universe.

For a forcing notion $\mathbb{P}$ , we say that $\mathbb{P}$ is
\textit{strongly arboreal} if the conditions of $\mathbb{P}$ are
perfect trees on $\omega$, and 
\[
T\in\mathbb{P},\ s\in T\Rightarrow T\restriction_{s}\in\mathbb{P}
\]
where $T\restriction_{s}=\{t\ :\ t\in T;\ t\supseteq s\ or\ t\subseteq s\}$.

\begin{thm}\cite{Ikegami_main}
\label{delta_1_2_and_L_generics}Let $I$ be a provably
ccc, provably $\mathbf{\Delta_{2}^{1}}$ and Borel generated $\sigma$-ideal such that $\mathbb{P}_{I}$ is strongly arboreal.  The following are equivalent:
\begin{enumerate}
\item Every $\mathbf{\Delta_{2}^{1}}$ set of reals is $I$-measurable.
\item For any real $z$ and $B\in\mathbb{P}_I,$ there is an $L[z]$ generic in $B$. 
\item $\Sigma^1_3$-$\mathbb{P}_I$-generic-absoluteness.
\end{enumerate}
\end{thm}

\begin{thm}\cite{Ikegami_main}
\label{sigma_1_2_and_L_generics}Let $I$ be a provably
ccc, provably $\mathbf{\Delta_{2}^{1}}$ and Borel generated $\sigma$-ideal such that $\mathbb{P}_{I}$ is strongly arboreal. The following are equivalent:
\begin{enumerate}
\item Every $\mathbf{\Sigma_{2}^{1}}$ set of reals is $\mathbb{P}_{I}$-measurable.
\item For any real $z$, the set of $\mathbb{P}_I$ generics over $L[z]$ is co-$I$.
\end{enumerate}
\end{thm}

\subsection{The results of this paper}

We devote section \ref{sec:---measruable}
for a detailed exposition of the notion of $I$-measurability,
where $I$ is any $\sigma$-ideal:

\begin{defn}
\label{def: I measurable-1}\cite{brendle_lowe} Let $I$ be a $\sigma$-ideal, and $A\subseteq\mathbb{R}.$
We say that $A$ is $I$-measurable if for every $B\in\mathbb{P}_{I}$,
there is $B'\subseteq B$ Borel $I$-positive such that either $B'\subseteq A$ or $B'\subseteq^{\sim}A$.
\end{defn}

This notion extends the notion of measurability for ccc $\sigma$
-ideals, as the following proposition shows. Recall that $I$ is \emph{Borel generated} if any $I$-small set is contained in a Borel $I$-small set.

\begin{prop}
If $I$ is ccc and Borel generated, then $A$ is $I$-measurable if and only if there
is $B$ Borel such that $A\triangle B\in I$. 
\end{prop}

A few basic facts on regularity of measurable sets are given, among
which:

\begin{prop}
If $A$ is universally Baire and $I$ is a proper $\sigma$-ideal, then $A$ is
$I$-measurable. 
\end{prop}

\begin{prop}
If there is a measurable cardinal and $I$ is a proper $\sigma$-
ideal, then $\mathbf{\Sigma_{2}^{1}}$ sets are $I$-measurable.
\end{prop}

We say that $I$ is \textit{provably ccc} if ``$I$ is ccc'' is
a theorem of $ZFC.$ We say that $I$ is $\Sigma_{n}^{1}$ or $\Pi_{n}^{1}$
if
\[
\{c\ :\ c\in BC,\ B_{c}\in I\}
\]
 is $\Sigma_{n}^{1}$ or $\Pi_{n}^{1}$ ,where $BC$ is the set of
Borel codes and for $c\in BC$, $B_{c}$ is the Borel set coded by
$c$.

\begin{prop}
Let $I$ be a $\mathbf{\Sigma^1_2}$ provably ccc $\sigma$-ideal. If $\omega_1$ is inaccessible to the reals, then $\mathbf{\Sigma^1_2}$ sets are $I$-measurable.
\end{prop}

In section \ref{sec:Generic--Absoluteness} we elaborate on equivalent formulations of measurability in terms of generic absoluteness and transcendence properties over $L$. This section is heavily based on ideas, proofs and arguments from  Ikegami \cite{ikegami_sacks} and \cite{Ikegami_main}, presented in a somewhat different context. Establishing those equivalences in the context we work in will prove useful in understanding the perfect set properties of equivalence relations with small classes. An overview of Ikegami's original results can be found in subsection \ref{subsection_ikegami} above. 

Recall that \emph{$\Sigma_{3}^{1}$-$\mathbb{P}_{I}$-generic-absoluteness} is the property that $\Sigma_{3}^{1}$ statements on ground model reals are absolute between the universe and $\mathbb{P}_I$-generic extensions of the universe.

We use the following notion due to Zapletal to establish an equivalence
between $\Sigma_{3}^{1}$-$\mathbb{P}_{I}$-generic absoluteness and
$\mathbf{\Delta_{2}^{1}}$ $I$-measurability for any proper $\sigma$
-ideal $I$:

\begin{defn}
(\cite{zapletal_book} 2.3.4) For $\Gamma$ a pointclass, we say that $\Gamma$
has $\mathbb{P}_{I}$-Borel uniformization if given $A\in\Gamma$
a subset of $(\omega^{\omega})^{2}$ with nonempty sections and $B\subseteq\omega^{\omega}$
$I$-positive, there is $B'\subseteq B$ Borel $I$-positive and a Borel
function $f\subseteq A$ with domain $B'$. 
\end{defn}

\begin{thm}
Let $I$ be a proper $\sigma$-ideal. The following are equivalent:
\begin{enumerate}
\item $\Sigma_{3}^{1}$-$\mathbb{P}_{I}$-generic-absoluteness.
\item $\mathbf{\Sigma_{2}^{1}}$ has $\mathbb{P}_{I}$-Borel uniformization.
\item $\mathbf{\Pi_{1}^{1}}$ has $\mathbb{P}_{I}$-Borel uniformization.
\item $\mathbf{\Delta_{2}^{1}}$ sets are $I$-measurable.
\end{enumerate}
\end{thm}

For definable enough provably ccc $\sigma$-ideals, an argument from \cite{Ikegami_main} adds transcendence over $L$ to the list of equivalent statements:

\begin{thm}
Let $I$ be a $\mathbf{\Sigma^1_2}$ provably ccc $\sigma$-ideal. Then the following are equivalent:
\begin{enumerate}
\item $\mathbf{\Delta_{2}^{1}}$ sets are $I$-measurable.
\item For every $B \in \mathbb{P}_I$ and for every real $z$, there is a $\mathbb{P}_I$-generic over $L[z]$ in $B$.
\end{enumerate}
\end{thm}

Sections \ref{sec:sigma_1_2 transcendence to PSP} and \ref{sec:sigma_1_2 psp to transencdence} focus on 
$PSP_{I}(\mathbf{\Delta_{2}^{1})}$,  $PSP_{I}(\mathbf{\Sigma_{2}^{1})}$ and problem \ref{problem}.
Section \ref{sec:sigma_1_2 transcendence to PSP} presents properties
of transcendence over $L$ which are sufficient  conditions for $PSP_{I}(\mathbf{\Sigma_{2}^{1})}$ when $I$ is provably ccc and definable enough. Section \ref{sec:sigma_1_2 psp to transencdence} completes the picture by providing necessary conditions for both $PSP_{I}(\mathbf{\Sigma_{2}^{1})}$
and $PSP_{I}(\mathbf{\Delta_{2}^{1})}$ for the same class of $\sigma$-ideals.

A set $A$ of reals is\emph{ a set of $\mathbb{P}_{I}*\dot{\mathbb{P}_{I}}$} generics if for every $x\in A$
and $y\in A$ which are not equal, $(x,y)$ is $\mathbb{P}_{I}*\dot{\mathbb{P}_{I}}$ generic.

\begin{thm}
Let $I$ be $\mathbf{\mathbf{\Sigma_{2}^{1}}}$
or $\mathbf{\Pi_{2}^{1}}$ and provably
ccc. If for any real $z$ and $B\in\mathbb{P}_{I}$ there is a perfect set $P \subseteq B$ of $\mathbb{P}_{I}*\dot{\mathbb{P}_{I}}$
generics over $L[z]$, then $PSP_{I}(\mathbf{\Sigma_{2}^{1})}.$
\end{thm}

Then together with another result of \cite{mutual_generics} on the existence of a perfect set
of $\mathbb{P}_{I}*\dot{\mathbb{P}_{I}}$ generics, we have:

\begin{cor}
Let $I$ be $\mathbf{\mathbf{\Sigma_{2}^{1}}}$ or $\mathbf{\Pi_{2}^{1}}$, provably
ccc, homogeneous and with the Fubini property. If $\omega_1$ is inaccessible to the reals then $PSP_{I}(\mathbf{\Sigma_{2}^{1})}.$
\end{cor}

We remark that stronger large cardinal assumptions clearly imply $PSP_{I}(\mathbf{\Sigma_{2}^{1}})$ and more. For example, if a measurable cardinal exists, then for any proper $\sigma$-ideal $I$, $PSP_{I}(\mathbf{\Sigma_{3}^{1}})$  and $PSP_{I}(\mathbf{\Pi_{3}^{1}})$ are true. To see why, reread section 2 of \cite{my_paper} and replace the last line of the proof of claim 2.6 by the argument from \cite{my_Borel_canonization} theorem 3.9.

As to the necessary conditions:

\begin{thm}
Let $I$ be $\mathbf{\Sigma_{2}^{1}}$ and provably ccc. $PSP_{I}(\mathbf{\Sigma_{2}^{1}})$
implies that for every $B \in \mathbb{P}_I$ and for every real $z$ there exists a perfect set $P \subseteq B$ of $\mathbb{P}_{I}$-generics
over $L[z]$.
\end{thm}

\begin{thm}
Let $I$ be  $\mathbf{\Sigma_{2}^{1}}$ and provably ccc. $PSP_{I}(\mathbf{\Delta_{2}^{1}})$
implies that for every $B \in \mathbb{P}_I$ and for every real $z$ there exists a $\mathbb{P}_{I}$-generic over $L[z]$ in $B$.
\end{thm}

At that point we can answer problem \ref{problem} (1) positively for a wide class of ccc $\sigma$-ideals:

\begin{cor}
Let $I$ be  $\mathbf{\Sigma_{2}^{1}}$ and provably ccc. The following are equivalent:
\begin{enumerate}
\item $PSP_{I}(\mathbf{\Delta_{2}^{1}})$.
\item $\mathbf{\Delta_{2}^{1}}$ sets are $I$-measurable.
\end{enumerate}
\end{cor}

However, problem \ref{problem} (2) has a negative answer:

\begin{cor}
The following are equivalent:
\begin{enumerate}
\item $PSP_{meager}(\mathbf{\Sigma_{2}^{1})}.$
\item $\mathbf{\Delta_{2}^{1}}$ sets have the Baire property.
\end{enumerate}
\end{cor}

Trying to characterize $PSP_{null}(\mathbf{\Sigma_{2}^{1})}$ leads to an open problem of Brendle (\cite{mutual_generics} 2.8):  is the existence of a perfect set of random reals equivalent to the existence of
a perfect set of mutually random reals? A positive answer will imply that $PSP_{null}(\mathbf{\Sigma_{2}^{1})}$ is equivalent to the existence of a perfect set of random reals over $L[z]$ for any $z$.

Another problem yet to be solved is characterizing $PSP_{countable}(\mathbf{\Sigma_{2}^{1})}.$
In light of theorem \ref{thm_countable_ideal)} and \cite{brendle_lowe} 7.1, we conjecture:

\begin{problem}
If for every $z$ there is $x\notin L[z]$, then $PSP_{countable}(\mathbf{\Sigma_{2}^{1})}$.
\end{problem}

\subsection{Acknowledgments}

This research was carried out under the supervision of Menachem Magidor,
and would not be possible without his elegant ideas and deep insights.
The author would like to thank him for his dedicated help. The author
would also like to thank Asaf Karagila for useful discussions, and
Amit Solomon for a fruitful and surprising collaboration around the
last section of this paper. 

\section{\label{sec:---measruable}$I$-measurable sets}

\begin{defn}
\label{def: I measurable}Let $I$ be a $\sigma$-ideal, and $A\subseteq\mathbb{R}.$
We say that $A$ is \emph{ $I$-measurable} if for every $B\in\mathbb{P}_{I}$,
there is $B'\subseteq B$ Borel $I$-positive such that either $B'\subseteq A$ or $B'\subseteq^{\sim}A$.
\end{defn}

\begin{rem}
The following are easy to observe:
\begin{enumerate}
\item Borel sets are measurable.
\item There is a non measurable set.
\item If $A$ is $I$-measurable, then either $A$ or $^{\sim}A$ contain
a Borel $I$-positive set.
\item If $I$ is such that $I$-positive sets contain $I$-positive Borel sets, then all
$I$-small sets are measurable. In that case, a set $A$ will be
$I$ small if and only if for every $B\in\mathbb{P}_{I}$ there is
$B'\subseteq B$ in $\mathbb{P}_{I}$ such that $B'\subseteq^{\sim}A.$
\end{enumerate}
\end{rem}

\begin{prop}
If $I$ is ccc and Borel generated, then $A$ is $I$-measurable if and only if there
is $B$ Borel such that $A\triangle B\in I$. Therefore, for $I$ ccc and Borel generated,
the $I$-measurable sets form a $\sigma$-algebra.
\end{prop}

\begin{proof}
First assume there is  a Borel set $B$ such that $A\triangle B \in I$, and fix $C\in I$ such that $A\triangle B \subseteq C$. Given a condition $D$, $D-C$ is $I$ positive and disjoint of $A \triangle B$. Hence, $(D-C)\cap B \subseteq A$ and $(D-C)\cap (^{\sim}B) \subseteq ^{\sim}A$.

For the other
direction, let $A$ be $I$-measurable. The set \[D=\{B\ :\ B \in \mathbb{P}_I;\ B\subseteq A\ or\ B\subseteq ^{\sim}A\}\] is dense -- let $B_{n}$ be a maximal antichain of elements of $D$. Define 
\[
B=\bigcup_{B_{n}'\subseteq A}B_{n}'
\]
and
\[
C=\bigcup_{B_{n}'\subseteq ^{\sim} A}B_{n}'.
\]
The complement $^{\sim}(B\cup C)$ must be $I$-small, otherwise we could extend the maximal antichain. We then claim that $B$ is the required approximation of $A$, since $B \subseteq A$ and
\[
A-B \subseteq ^{\sim}(B\cup C)\]
which is $I$-small.
\end{proof}

The last proposition shows that definition \ref{def: I measurable}
coincides with the traditional definition for a wide class of ccc ideals.
For many other examples of $\sigma$-ideals, definition \ref{def: I measurable}
is not new as well. For example, for the case of the countable ideal the notion
of Sacks measurability is well known for years. A set $A\subseteq\mathbb{R}$
is Sacks measurable if for any perfect set $P$ there is $P'\subseteq P$
perfect such that $P'\subseteq A$ or $P'\subseteq^{\sim}A$ -- exactly
the same as the definition discussed here.

We list related regularity properties of universally Baire and projective
sets. $M\preceq H_\theta$ always mean that $M$ is a countable elementary submodel of a large enough $H_\theta$.

\begin{prop}
\label{UB_measurability}
If $A$ is universally Baire and $I$ is a proper $\sigma$-ideal, then $A$ is
$I$-measurable. Furthermore, some $B\in\mathbb{P}_{I}$ forces
$x_{gen}\in A$ if and only if $A$ contains a Borel $I$-positive
subset.
\end{prop}

\begin{proof}
Let $B\in\mathbb{P}_{I}$ and let $A,B\in M\preceq H_{\theta}$ a countable elementary submodel of a large enough $H_\theta$ . Let $B'\subseteq B$
be such that $B'\Vdash x_{gen}\in A$ or $B'\Vdash x_{gen}\notin A$. Let
$B''\subseteq B'$ be the set of $M$-generics. 
Using the universally Baire definition of $A$ we find that $B''\subseteq A$
or $B''\subseteq^{\sim}A$.
\end{proof}

\begin{cor}
Let $A$ be universally Baire and $I$ a ccc $\sigma$-ideal. Then $A$ is either contained in a Borel
$I$-small set, or contains a Borel $I$-positive set.
\end{cor}

\begin{proof}
If some condition forces $x_{gen}\in A$, then by the last proposition, $A$ contains a Borel
$I$-positive set. Otherwise, $\Vdash_{\mathbb{P}_{I}}x_{gen}\in^{\sim}A$.
Hence the $M$-generics are all in $^{\sim}A$, so $^{\sim}A$ contains
a Borel co-$I$ set, and $A$ is contained in a Borel $I$-small set.
\end{proof}

\begin{prop}
\label{measurable_implies_sigma_1_2_measurable}
If there is a measurable cardinal and $I$ is a proper $\sigma$-
ideal, then $\mathbf{\Sigma_{2}^{1}}$ sets are $I$-measurable.
\end{prop}

\begin{proof}
Let $A$ be $\Sigma_{2}^{1}$ and $B\in\mathbb{P}_{I}$. We may assume
that $B\Vdash x_{gen}\in A$ or $B\Vdash x_{gen}\notin A.$ Now let
$M\preceq H_{\theta}$ contain all the relevant information and the
measurable cardinal. Let $B'\subseteq B$ be the set of $M$-generics
in $B$. There are 2 cases:

\begin{itemize}
\item $B\Vdash x_{gen}\in A$: We show that $B'\subseteq A$. Indeed, let
$x\in B'.$ Then $M[x]\models x\in A$, hence $x\in A$. 
\item $B\Vdash x_{gen}\notin A$: We show that $B'\subseteq^{\sim}A.$ Indeed,
let $x\in B'$. Then \[M[x]\models x\notin A.\] We claim that $x\notin A$. The argument is as in \cite{my_Borel_canonization} theorem 3.9. 
Assume otherwise -- $x\in A$ -- so there is some $\alpha<\omega_{1}$ such that
inner models in which $\alpha$ is countable think that $x\in A$.
For ease of notation, we let $N=M[x]$, and iterate $N$ uncountably
many times, so that $N_{\omega_{1}}$ will contain all countable ordinals.
In $N_{\omega_{1}}[coll(\omega,\alpha)],$ $x\in A,$ and using Shoenfield's
absoluteness, \[N_{\omega_{1}}\models x\in A\] as well. But $N_{\omega_{1}}$
is an elementary extension of $N=M[x]$ , in contradiction with
$M[x]\models x\notin A.$
\end{itemize}
\end{proof}

\begin{prop}
\label{omega1_inaccessible_implies_sigma_1_2_measurable}
Let $I$ be a $\mathbf{\Sigma^1_2}$ provably ccc $\sigma$-ideal. If $\omega_1$ is inaccessible to the reals, then $\mathbf{\Sigma^1_2}$ sets are $I$-measurable.
\end{prop}

\begin{proof}
Let $A$ be $\Sigma^1_2$, and $B\in\mathbb{P}_I$. Extend $B$ to $B'$ forcing $x_{gen}\in A$ or $x_{gen}\notin A$. The first case is exactly the same as the first case in proposition \ref{measurable_implies_sigma_1_2_measurable}. For the second case, let $M \preceq H_{\theta}$ be a countable elementary submodel containing all the relevant information and $L_{\omega_1^L}$, and in particular containing all constructible reals. It will be enough to show that the $M$-generics in $B'$ are elements of $^{\sim}A$. Indeed, if \[M[x] \models x_{gen}\in ^{\sim}A\] then \[L_{\omega_1^L}[x]\models x_{gen}\in ^{\sim}A\] by analytic absoluteness only. Since $x$ is generic over $L$ (using the assumptions on $I$), $\omega_1^L=\omega_1^{L[x]}$ and we can use Shoenfield's absoluteness to reflect the last statement to $\mathbb{V}$ and complete the proof.
\end{proof}

\begin{rem}
In fact, a sufficient assumption on the $\sigma$-ideal $I$ is
that for every $z$, $I\cap L[z]\in L[z]$ , and $L[z]\models I\cap L[z]\ is\ ccc.$
\end{rem}

\begin{prop}
\label{condition_forces_sigma_1_2}
Let $A$ be $\mathbf{\Sigma_{2}^{1}}$ and $I$ a proper $\sigma$-ideal. If some $B\in\mathbb{P}_{I}$ forces $x_{gen}\in A$ then
$A$ contains a Borel $I$-positive subset.
\end{prop}

\begin{proof}
$A$ can be represented as a union of $\aleph_1$ Borel sets: \[A=\bigcup_{\alpha<\omega_{1}}B_{\alpha}.\] Let $C\Vdash x_{gen}\in A$.
Then there is $C'\subseteq C$ such that $C'\Vdash x_{gen}\in B_{\alpha}$
-- where we have used the assumption that $\omega_1$ is preserved. $B_{\alpha}$ then must be $I$-positive.
\end{proof}

The rest of this section is concerned only with ccc $\sigma$-ideals. Both of the following are false for
general proper $\sigma$-ideals -- consider the countable ideal and the $\Pi_{1}^{1}$
set with no perfect subset.

\begin{prop}
Let $A$ be $\mathbf{\Pi_{2}^{1}}$ and $I$ a ccc $\sigma$-
ideal. If $A$ is $I$-positive then there is some $B\in\mathbb{P}_{I}$
forcing $x_{gen}\in A$.
\end{prop}
\begin{proof}

Assume \[\mathbb{P}_I \Vdash x_{gen}\notin A.\]As before, $^{\sim} A=\cup_{\alpha<\omega_{1}}B_{\alpha}$. Find a maximal antichain
forcing $x_{gen}\in B_{0}$, extend it to a maximal antichain forcing
$x_{gen}\in B_{0}\cup B_{1},$ and so on. Since antichains are countable,
the process must stop at a countable level. The union of all conditions in that antichain is a Borel
set $B$ contained modulo $I$ in $^{\sim} A$ such that \[\mathbb{P}_{I} \Vdash x_{gen}\in B.\]
The complement of $B$ must then be $I$-small. $A$ is contained in $^{\sim}B$ modulo $I$, therefore it is $I$-small as well -- which is what we wanted to show. 
\end{proof}

\begin{cor}
Let $A$ be $\mathbf{\Delta_{2}^{1}}$ and $I$ a ccc $\sigma$-ideal.
If $A$ is $I$-positive then $A$ contains a Borel $I$-positive
subset. In particular, a coanalytic set with no perfect subset is
$I$-small with respect to any ccc $\sigma$-ideal.
\end{cor}

\begin{proof}
Follows of the last two propositions.
\end{proof}

\section{\label{sec:Generic--Absoluteness}Measurability, Generic 
Absoluteness and Transcendence over $L$}

The main result of the following section establishes equivalences between three notions:
generic absoluteness, measurability -- as discussed in the previous section -- 
and the following notion due to Zapletal.

\begin{defn}
(\cite{zapletal_book} 2.3.4) For $\Gamma$ a pointclass, we say that $\Gamma$
has \emph{$\mathbb{P}_{I}$-Borel uniformization} if given $A\in\Gamma$
a subset of $(\omega^{\omega})^{2}$ with nonempty sections and $B\subseteq\omega^{\omega}$
$I$-positive, there is $B'\subseteq B$ Borel $I$-positive and a Borel
function $f\subseteq A$ with domain $B'$. 
\end{defn}

\begin{thm}
\label{TFAE}
Let $I$ be a proper $\sigma$-ideal. The following are equivalent:
\begin{enumerate}
\item $\Sigma_{3}^{1}$-$\mathbb{P}_{I}$-generic-absoluteness.
\item $\mathbf{\Sigma_{2}^{1}}$ has $\mathbb{P}_{I}$-Borel uniformization.
\item $\mathbf{\Pi_{1}^{1}}$ has $\mathbb{P}_{I}$-Borel uniformization.
\item $\mathbf{\Delta_{2}^{1}}$ sets are $I$-measurable.
\end{enumerate}
\end{thm}

\begin{proof}
$(1)\Rightarrow(2)$: Let $A$ be $\Sigma_{2}^{1}$ with nonempty
sections -- $\forall x\exists y\ (x,y)\in A$ -- which is $\Sigma_{3}^{1}$
and hence by assumption is preserved in $\mathbb{P}_{I}$ generic
extensions. Then there is a name $\tau$ such that \[B\Vdash(x,\tau)\in A.\]Let $f:C\to Y$ be a Borel function in the ground model such that
$C\Vdash(x,f(x))\in A$. If $M\preceq H_\theta$ is a countable elementary submodel containing all the relevant information and $x$ is $M$ generic, then $M[x]\models(x,f(x))\in A$ and using
$\Pi_{1}^{1}$ absoluteness, $(x,f(x))\in A$.

$(3)\Rightarrow(4):$ Let $B\in\mathbb{P}_{I}$ and $A$ a $\Delta_{2}^{1}$
set. Fix $C$ and $D$ $\Pi_{1}^{1}$ subsets of the plane such that
$\Pi(C)$, the projection of $C$, is $A$, and $\Pi(D)=^{\sim}A$. Since $C\cup D$ is a $\Pi_{1}^{1}$
set with nonempty sections, there is $B'\subseteq B$ in $\mathbb{P}_{I}$
and $f$ a Borel function such that
\[
\forall x\in B':\ (x,f(x))\in C\cup D.
\]
It follows that for $x \in B'$: \[x\in A \Leftrightarrow (x,f(x))\in C \Leftrightarrow (x,f(x))\notin D,\] so $B' \cap A$ is Borel. The same argument works for $B' \cap ^{\sim}A$. One of $B' \cap A, B'\cap ^{\sim}A$ must be $I$-positive.
$(4)\Rightarrow(1)$: We use the notation of \cite{Ikegami_main} and follow the proof of \cite{Ikegami_main} theorem 4.1 and claim 4.2.

Assume all $\mathbf{\Delta_{2}^{1}}$ sets are
$I$-measurable. We show that all $\mathbf{\Delta_{2}^{1}}$ sets
are $\mathbb{P}_{I}$-Baire, and that will be enough (See \cite{Ikegami_main}
3.9 and \cite{UB_sets}). 

Let $f:st(\mathbb{P}_{I})\to\omega^{\omega}$ be a Baire measurable
function and $A$ a $\mathbf{\Delta_{2}^{1}}$ set. It will be enough
to show that
\[
\{B\ :\ O_{B}\cap f^{-1}(A)\ meager\ or\ O_{B}-f^{-1}(A)\ meager\}
\]
is a dense set in $\mathbb{P}_{I}$ , where $O_{B}$ is $\{G\in st(\mathbb{P}_{I})\ :\ B\in G\}$.
Indeed, let $B\in\mathbb{P}_{I}.$ There is a name $\tau$ such that
for comeagerly many $G\in st(\mathbb{P}_{I})$:
\[
f(G)=\tau[G].
\]
Since $I$ is proper, there is $B'\subseteq B$ in $\mathbb{P}_{I}$
and $g:B'\to\omega^{\omega}$ Borel such that
\[
B'\Vdash g(x_{G})=\tau[G],
\]
which means that for comeagerly many $G\in st(\mathbb{P}_{I})$ such
that $B'\in G$, \[g(x_{G})=\tau[G]=f(G).\] Since $g^{-1}(A)$ is $\mathbf{\Delta_{2}^{1}},$
it is measurable by our assumption. Let $B''\subseteq B'$ in $\mathbb{P}_{I}$
be such that
\[
B''\subseteq g^{-1}(A)
\]
or 
\[
B''\subseteq g^{-1}(^{\sim}A).
\]
We continue with the 1st case -- the 2nd is similar. Since $B''\in G$
implies $B'\in G$, we conclude that for comeagerly many $G\in st(\mathbb{P}_{I})$
such that $B''\in G$
\[
g(x_{G})=\tau[G]=f(G)\in A
\]
whereas $f(G)\in A$ because $x_{G}\in B''$ and $B''\subseteq g^{-1}(A).$ That shows that $O_{B''}-f^{-1}(A)$
is meager.\end{proof}
We give here another argument for $(3)\Rightarrow(1)$ which we find interesting on its own. It is based on an argument from the proof of  \cite{ikegami_sacks} theorem 3.1:

$(3)\Rightarrow(1):$ By way of contradiction, assume $\forall x\neg\Psi(x)$
but $\Vdash\exists x\Psi(x)$ ,where $\Psi(x)=\forall y\Phi(x,y)$
and $\Phi$ is $\Sigma_{1}^{1}.$ Fix $B\in\mathbb{P}_{I}$ and $f \in \mathbb{V}$
a Borel function such that
\[
B\Vdash\Psi(f(x_{gen})).
\]
In $\mathbb{V}$
\[
\forall x\exists y\neg\Phi(f(x),y)
\]
so we can use $\Pi_{1}^{1}$ $\mathbb{P}_{I}$-Borel uniformization
to produce a Borel function $g:B'\to\omega^{\omega}$, $B'\subseteq B$
in $\mathbb{P}_{I},$ such that 
\[
\forall x\in B':\ \neg\Phi(f(x),g(x)).
\]
Since the last statement is $\Pi_{1}^{1}$ , it is preserved in generic
extensions. In particular, $B'\Vdash\neg\Phi(f(x_{gen}),g(x_{gen}))$,
whereas $B\Vdash\Psi(f(x_{gen}))=\forall y\Phi(f(x_{gen}),y)$ -- a
contradiction.
\\

For $\mathbf{\Sigma^1_2}$ provably ccc $\sigma$-ideals, we can add transcendence over $L$ to the list of equivalent conditions. This is no more than adapting \cite{Ikegami_main} Theorem 4.3 to our context, with a slight change in statement and almost no change in the proof.

\begin{thm}
\label{prop_measurability_vs_transcendence_for_ccc}
Let $I$ be a $\mathbf{\Sigma^1_2}$ provably ccc $\sigma$-ideal. Then the following are equivalent:
\begin{enumerate}
\item $\mathbf{\Delta_{2}^{1}}$ sets are $I$-measurable.
\item For every $B \in \mathbb{P}_I$ and for every real $z$, there is a $\mathbb{P}_I$-generic over $L[z]$ in $B$.
\end{enumerate}
\end{thm}

\begin{proof}

For $(1) \Rightarrow (2)$, we already know that $(1)$ implies $\Sigma^1_3$ $\mathbb{P}_I$-generic absoluteness. The set of $L[z]$ generics in $B$ is $\Pi^1_2(z)$ in this case, and can be forced to be nonempty. The conclusion follows.

As to $(2) \Rightarrow (1)$, using corollary \ref{omega1_inaccessible_implies_sigma_1_2_measurable}
we may assume there is a real $z$ such that $\omega_1^{L[z]}=\omega_1$. Let $A$ be $\Delta^1_2(a)$ and $B\in\mathbb{P}_I$. For ease of notation, let us assume that $z$, $a$ and $B$ are all constructible, so that we can work in $L$ and assume $\omega_1^L=\omega_1$.

We now decompose both $B\cap A$ and $B-A$ into $\aleph_1$ Borel sets, as both are $\Sigma^1_2$ sets. The decomposition is absolute between $L$ and $\mathbb{V}$, since they both agree on the first uncountable ordinal. In particular, all those Borel sets are constructible. By assumption, there is a generic over $L$ in $B$, which is, one of those Borel sets has an element which is $L$ generic. It follows that this set is $I$-positive in $L$. Our definability assumption on $I$ obligates it to be $I$-positive in $\mathbb{V}$ as well, and the proof is completed.
\end{proof}

We conclude the section with two remarks on $\mathbb{P}_I$-Borel uniformization.

\begin{rem}
The notion of $\mathbb{P}_I$-Borel uniformization is related to the notion of Borel canonization of Kanovei, Sabok and Zapletal \cite{ksz}:

If $\mathbf{\Pi_{2}^{1}}$ has $\mathbb{P}_I$-Borel uniformization then there is
Borel canonization of analytic equivalence relations.
\end{rem}

\begin{proof}
We use the rank defined in \cite{my_Borel_canonization} section 3. Let $(x,f)\in A$ if and only if $f\in WO$ and $\delta(x)\leq f$. $A$ is $\Pi^1_2$.
Since all classes are Borel, the sections of $A$ are nonempty, so we can use $\mathbb{P}_I$-Borel uniformization and find $B\in\mathbb{P}_{I}$ and $f:B\to WO$ Borel such that
$\delta(x)\leq f(x).$ The boundedness theorem completes the argument.
\end{proof}

\begin{rem}
The following are equivalent:
\begin{enumerate}
\item $\Sigma_{3}^{1}$-$\mathbb{P}_I$-generic absoluteness.
\item Given $\Phi(x,y,\bar{z})$ a $\Pi_{1}^{1}$ formula, the statement ''$\Phi(x,y,\bar{z})$ has nonempty sections'' is absolute between $\mathbb{P}_I$-generic extensions.
\item Given $\Phi(x,y,\bar{z})$ a $\Pi_{1}^{1}$ formula, the statement ''$\Phi(x,y,\bar{z})$ is a graph of a function'' is absolute between $\mathbb{P}_I$-generic extensions.
\end{enumerate}
\end{rem}

\begin{proof}
Above results and $\Pi_{1}^{1}$ uniformization.
\end{proof}

\section{\label{sec:sigma_1_2 transcendence to PSP}From transcendence over
$L$ to $PSP_{I}\mathbf{\Sigma_{2}^{1}}$}

In the following section we find transcendence properties over $L$ which are sufficient conditions for $PSP_{I}(\mathbf{\Sigma_{2}^{1})}$. 

\begin{thm}
\label{main_thm_Sigma_1_2 mutually_generics}Let $I$ be $\mathbf{\mathbf{\Sigma_{2}^{1}}}$
or $\mathbf{\Pi_{2}^{1}}$ and provably
ccc. If for any real $z$ and $B\in\mathbb{P}_{I}$ there is a perfect set $P \subseteq B$ of $\mathbb{P}_{I}*\dot{\mathbb{P}_{I}}$
generics over $L[z]$, then $PSP_{I}(\mathbf{\Sigma_{2}^{1})}.$
\end{thm}

We will say that \emph{$I$ is homogeneous} if $\mathbb{P}_I$ is a weakly homogeneous forcing notion. The forcing notion \cite{mutual_generics} 2.6 has natural counterparts for any $\sigma$-ideal with the Fubini property, leading to the following corollary:

\begin{cor}
Let $I$ be $\mathbf{\mathbf{\Sigma_{2}^{1}}}$ or $\mathbf{\Pi_{2}^{1}}$, provably
ccc, homogeneous and with the Fubini property. If $\omega_1$ is inaccessible to the reals then $PSP_{I}(\mathbf{\Sigma_{2}^{1})}.$
\end{cor}

See corollaries \ref{meager case} and \ref{null case} for the application of theorem \ref{main_thm_Sigma_1_2 mutually_generics} on the meager and null ideals. Note that although the last corollary can be applied to those ideals, it does not produce any new result -- when $\omega_1$ is inaccessible to the reals, $\mathbf{\Sigma^1_2}$ sets have the Baire property and are Lebesgue measurable, and Mycielski's theorems \ref{mycielski_meager} and \ref{Mycielski_measure_zero} are valid.

\begin{proof}
(of theorem \ref{main_thm_Sigma_1_2 mutually_generics}) Let $E$ be a $\mathbf{\Sigma_{2}^{1}}$ equivalence relation on $B$ Borel $I$-positive with
$I$-small classes. We may assume $E$ is lightface $\Sigma^1_2$ and $B$ is constructible.

We first claim that the generic added by forcing
with $B$ belongs to a new $E$-class. Otherwise,
fix $z\in\mathbb{V}$ and $B'\subseteq B$ such that
\[
B'\Vdash x_{G}\in[z].
\]
Let $M$ be an elementary submodel of the universe containing $z$
and all the relevant information. Let $x\in B'$ be $M$-generic.
Then $M[x]\models xEz$, and so by $\mathbf{\Pi^1_1}$ absoluteness, $\mathbb{V}\models xEz.$
We have thus shown that all the $M$-generics in $B'$ are in the
equivalence class of $z$, hence $[z]$ is $I$-positive -- a contradiction. 

Consider the two-step iteration $\mathbb{P}_{I}*\dot{\mathbb{P}_{I}}$.

\begin{claim} If $B_1\subseteq B$ and $L \models B_1\Vdash \dot{B_2}\subseteq B$ then $L\models(B_{1},\dot{B_{2}})\nVdash(z_{1}Ez_{2})$, where $z_1,z_2$ are the $\mathbb{P}_I$-generics.
\end{claim}

\begin{proof}
The idea is similar to the one of the proof of \cite{foreman_magidor} theorem 3.4. Note that the first generic we will mention is an $L$-generic that is an element of $\mathbb{V}$, while the second one is a real $\mathbb{V}$-generic.

Assume otherwise, and let $(B_{1},\dot{B_{2}})\in\mathbb{P}_{I}*\dot{\mathbb{P}_{I}}$
be as above and such that $L\models(B_{1},\dot{B_{2}})\Vdash z_{1}Ez_{2}$ . Let
\[
z_{1}\in\mathbb{V}
\]
 be $\mathbb{P}_{I}$-generic over $L$ such that $z_{1}\in B_{1}.$
Then $\dot{B_{2}}[z_{1}]$, the interpretation of $\dot{B_{2}}$
by the generic filter of $z_{1},$ is an $I$-positive Borel set
in $L[z_{1}]$ -- we denote it by $B_{2}$ - which is a subset of $B$. Let $z_{2}\in B_{2}$ be
$\mathbb{P}_{I}$ generic over $\mathbb{V}.$ Then $z_{2}$ is also
$\mathbb{P}_{I}$ generic over $L[z_{1}]$. By the assumption
\[
L[z_{1}][z_{2}]\models z_{1}Ez_{2}
\]
 and hence $\mathbb{V}\models z_{1}Ez_{2}.$ However, we have shown
that the $\mathbb{P}_{I}$-generic $z_{2} \in B$ cannot be an element
of the ground model set $[z_{1}]_{E}$ -- a contradiction.
\end{proof}

It follows that in $L$, \[D=\{(B_{1},\dot{B_{2}})\ :\ \neg(B_1\subseteq B \wedge B_1 \Vdash \dot{B_2}\subseteq B)\ or\ (B_{1},\dot{B_{2}})\Vdash\neg(z_1 E z_2)\}\]
is dense in $\mathbb{P}_{I}*\dot{\mathbb{P}_{I}}$. Therefore, given
$(x,y)\in B^2$ which is $\mathbb{P}_{I}*\dot{\mathbb{P}_{I}}$ generic over $L$,
\[
L[x][y]\models\neg(xEy)
\]
which together with Shoenfield's absoluteness implies that $x$ and
$y$ are inequivalent. Since we assumed there is a perfect set $P \subseteq B$ of
$\mathbb{P}_{I}*\dot{\mathbb{P}_{I}}$ generics over $L$, that concludes
the proof.
\end{proof}

\section{\label{sec:sigma_1_2 psp to transencdence}From $PSP_{I}\mathbf{\Sigma_{2}^{1}}$
to transcendence over $L$}

In the following section we find necessary conditions for $PSP_{I}(\mathbf{\Sigma_{2}^{1}})$ and $PSP_{I}(\mathbf{\Delta_{2}^{1}})$, for $I$ a $\mathbf{\Sigma_{2}^{1}}$ and provably ccc $\sigma$-ideal. The author wishes to thank Amit Solomon for his help with obtaining the following two results.

\begin{thm}
\label{necessary_sigma_1_2}
Let $I$ be $\mathbf{\Sigma_{2}^{1}}$ and provably ccc. $PSP_{I}(\mathbf{\Sigma_{2}^{1}})$
implies that for every $B \in \mathbb{P}_I$ and for every real $z$ there exists a perfect set $P \subseteq B$ of $\mathbb{P}_{I}$-generics
over $L[z]$.
\end{thm}

\begin{proof}
Since $PSP_{I}(\mathbf{\Sigma_{2}^{1})}$ clearly implies $PSP_{I}(\mathbf{\Delta_{2}^{1})}$ and $PSP_{countable}(\mathbf{\Delta_{2}^{1})}$, theorem \ref{thm_countable_ideal)} and \cite{ikegami_sacks} guarantees that Sacks forcing preserves $\Sigma_{3}^{1}$ statements. A
perfect set in $B$ of $\mathbb{P}_{I}$-generics over $L$ exists iff
\[\exists P\subseteq B \ perfect\ (\forall x\in P\ \forall c\ \left((c\in L)\wedge(c\in BC)\wedge(B_{c}\in I)\Rightarrow x\notin B_{c}\right).\]
That is a $\Sigma_{3}^{1}$ statement, hence if Sacks forcing adds
a perfect subset of $B$ of $\mathbb{P}_I$-generics, we are done. 

We will find a condition $P\subseteq B$
in Sacks forcing such that any new real added to $P$ must be $\mathbb{P}_{I}$-generic. Since Sacks forcing adds a perfect set of new reals to the
condition $P$, that will be enough.

The first stage is defining a $\Sigma^1_2$ equivalence relation on $B$ whose classes are either $I$-small sets in $L$, or singletons which are $\mathbb{P}_I$-generic elements over $L$.

Let $I_{L}(c)$ be the statement \[(c\in L)\wedge(c\in BC)\wedge(B_{c}\in I).\] Let
$D(x,c)$ be the statement: $I_{L}(c)$, $x \in B_c$ and 
\[
\forall c' \ I_{L}(c') \wedge x \in B_{c'} \Rightarrow(c\leq_{L}c'),
\]
which is, $c$ is the 1st $I$-small set in $L$ that has $x$ as
one of its elements. Note that $D(x,c)$ is $\Sigma_{2}^{1}$
since it can be decided inside a large enough countable model. We then consider
the following $\Sigma_{2}^{1}$ equivalence relation on $B$:
\[
\forall x,y\in B: xEy\Leftrightarrow(x=y)\vee\ \exists c\ (D(x,c)\wedge D(y,c)).
\]
Under $E$, the $\mathbb{P}_{I}$-generics over $L$ form equivalence
classes that are singletons. The rest of the classes are all contained
in an $I$-small set of $L$, hence are $I$-small. Since all
classes are $I$-small, $PSP_{I}(\mathbf{\Sigma_{2}^{1})}$ implies
the existence of a perfect set $P \subseteq B$ of pairwise inequivalent elements. 

We first show that any new Sacks real in $P$ must belong to a new $E$-class. By Shoenfield's
absoluteness, $P$ remains a perfect set of pairwise inequivalent elements in the $\mathbb{P}_I$-generic extension.
In addition, if the class of $z\in\mathbb{V}$ had no representative
in $P$ -
\[
\forall x\ x\in P\rightarrow\neg(xEz)
\]
- then $P\cap[z]_{E}$ will remain empty in the generic extension as
well. Therefore the new Sacks real in $P$ must indeed belong to a new class.

We can now complete the proof by showing a Sacks real in $P$ is $\mathbb{P}_I$-generic over $L$. Indeed, if it hadn't been, there would be $c\in L$ such that $D(x_{gen},c)$, so after forcing  \[\exists y\in B D(y,c).\] That is a $\Sigma^1_2$ statement, therefore true in the ground model as well. It means that $x_{gen}$ is an element of a ground model class, which is a contradiction.
\end{proof}

\begin{thm}
Let $I$ be  $\mathbf{\Sigma_{2}^{1}}$ and provably ccc. $PSP_{I}(\mathbf{\Delta_{2}^{1}})$
implies that for every $B \in \mathbb{P}_I$ and for every real $z$ there exists a $\mathbb{P}_{I}$-generic over $L[z]$ in $B$.
\end{thm}

\begin{proof}
Let $D$ and $E$ be defined as in the previous proof, and $ZFC^*$ a large enough finite fragment of $ZFC$. We define another equivalence relation
which is $\Pi_{2}^{1}:$ For $x,y\in B$, $xFy$ if and only if
\[
\forall M\ \left((M\in WO)\wedge(M\models ZFC^{*})\wedge(x,y\in M)\wedge(\exists c\ M\models D(x,c)\right)\Rightarrow(M \models D(y,c))
\]
and 
\[
\forall M\ \left((M\in WO)\wedge(M\models ZFC^{*})\wedge(x,y\in M)\wedge(\exists c\ M\models D(y,c)\right)\Rightarrow(M \models D(x,c)).
\]
If $x$ and $y$ are not $\mathbb{P}_{I}$-generics over $L$, then
$xEy\Leftrightarrow xFy$. The equivalence relations $E$ and $F$ are only different
on the set of the $\mathbb{P}_{I}$-generics over $L$: under $E$, the
$\mathbb{P}_{I}$-generics over $L$ form equivalence classes that
are singletons, whereas under $F$ they form one equivalence class.

By way of contradiction, assume that in $B$ there are no $\mathbb{P}_{I}$-generics
over $L$. Then $E$ and $F$ coincide, and $E$ becomes $\Delta_{2}^{1}$.
$PSP_{I}(\mathbf{\Delta_{2}^{1}})$ then guarantees the existence of a perfect set $P$ of pairwise inequivalent elements. We continue just as before -- recall that for the Sacks $\Sigma^1_3$ generic absoluteness we only used $PSP_{I}(\mathbf{\Delta^1_2})$. We get a perfect set $P \subseteq B$ of $\mathbb{P}_I$-generics over $L$ -- a contradiction.
\end{proof}

We can finally answer problem \ref{problem} (1) for a wide class of ccc $\sigma$-ideals:

\begin{cor}
Let $I$ be  $\mathbf{\Sigma_{2}^{1}}$ and provably ccc. The following are equivalent:
\begin{enumerate}
\item $PSP_{I}(\mathbf{\Delta_{2}^{1}})$.
\item $\mathbf{\Delta_{2}^{1}}$ sets are $I$-measurable.
\end{enumerate}
\end{cor}

\begin{proof}
$(1) \Rightarrow (2)$ is the previous theorem together with proposition \ref{prop_measurability_vs_transcendence_for_ccc}. $(2) \Rightarrow (1)$ is proposition \ref{TFAE} and theorem \ref{previous_delta_1_2}.
\end{proof}

However, problem \ref{problem} (2) has a negative answer:

\begin{cor}\label{meager case}
The following are equivalent:
\begin{enumerate}
\item $PSP_{meager}(\mathbf{\Sigma_{2}^{1})}.$
\item $\mathbf{\Delta_{2}^{1}}$ sets have the Baire property.
\end{enumerate}
\end{cor}

\begin{proof}
Recall that for $I=meager$, if there is a $\mathbb{P}_{I}$ generic
over $L[z]$, then for every $B \in \mathbb{P}_I$ there is a perfect set $P\subseteq B$ of $\mathbb{P}_{I}*\dot{\mathbb{P}_{I}}$
generics over $L[z]$ (see
\cite{mutual_generics} 1.1).
\end{proof}

We summarize what we know about the case of the null ideal:

\begin{cor}
\label{null case}
Let $I$ be the null ideal:
\begin{enumerate}
\item $PSP_{I}(\mathbf{\Sigma_{2}^{1})}$ implies the existence of a perfect
set of random reals over $L[z]$, for any real $z$.
\item If for any real $z$ there is a perfect set of $\mathbb{P}_{I}*\dot{\mathbb{P}_{I}}$
random reals over $L[z]$, then $PSP_{I}(\mathbf{\Sigma_{2}^{1}})$.
\end{enumerate}
\end{cor}

\begin{proof}
Random real forcing is weakly homogeneous, so random reals if exist, exist in every Borel set of positive measure.
\end{proof}

If the existence of a perfect set of random reals is equivalent to the existence of
a perfect set of mutually random reals then both conditions above are equivalent -- that
is an open question, see \cite{mutual_generics} 2.8.

We do not know the status of problem \ref{problem} (2) for the case of the null ideal - can we have $PSP_{null}(\mathbf{\Sigma_{2}^{1}})$ with a $\mathbf{\Sigma^1_2}$ set which is not Lebesgue measruable?

\begin{problem}
Is it consistent to have a perfect set of $\mathbb{P}_{I}*\dot{\mathbb{P}_{I}}$
random reals over $L[z]$ for any real $z$, and a $\mathbf{\Sigma^1_2}$ set which is not Lebesgue measurable? 
\end{problem}

Although the countable ideal is not ccc and hence out of the scope
of the last 2 sections, we still find ourselves very curious about understanding
$PSP_{countable}(\mathbf{\Sigma_{2}^{1})}$. In light of theorem \ref{thm_countable_ideal)} and \cite{brendle_lowe} 7.1, we conjecture that:

\begin{problem}
If for every $z$ there is $x\notin L[z]$, then $PSP_{countable}(\mathbf{\Sigma_{2}^{1})}$
.
\end{problem}

\selectlanguage{american}%

\end{onehalfspace}
\end{document}